\documentclass[12pt,reqno]{article}

\usepackage[usenames]{color}
\usepackage{amssymb}
\usepackage{amsmath}
\usepackage{amsthm}
\usepackage{amsfonts}
\usepackage{amscd}
\usepackage{graphicx}

\usepackage[colorlinks=true,
linkcolor=webgreen,
filecolor=webbrown,
citecolor=webgreen]{hyperref}

\definecolor{webgreen}{rgb}{0,.5,0}
\definecolor{webbrown}{rgb}{.6,0,0}

\usepackage{color}
\usepackage{fullpage}
\usepackage{float}

\usepackage{graphics}
\usepackage{latexsym}
\usepackage{epsf}

\newcommand{\blue}[1]{{\color{blue}#1}}
\newcommand{\red}[1]{{\color{red}#1}}

\begin{document}

\theoremstyle{plain}
\newtheorem{theorem}{Theorem}
\newtheorem{corollary}[theorem]{Corollary}
\newtheorem{lemma}[theorem]{Lemma}
\newtheorem{proposition}[theorem]{Proposition}

\theoremstyle{definition}
\newtheorem{definition}[theorem]{Definition}
\newtheorem{example}[theorem]{Example}
\newtheorem{conjecture}[theorem]{Conjecture}

\theoremstyle{remark}
\newtheorem*{remark}{Remark}

\begin{center}
\vskip 1cm{\bf  
\large On the sum of the sixth powers of Fibonacci numbers}
\bigskip

Kunle Adegoke \\ 
Department of Physics and Engineering Physics\\
Obafemi Awolowo University, Ile-Ife, Nigeria \\
\href{mailto:adegoke00@gmail.com }{\tt adegoke00@gmail.com}\\
\ \\
Olawanle Layeni\\
Department of Mathematics\\
Obafemi Awolowo University, Ile-Ife, Nigeria\\
\href{mailto:olayeni@oauife.edu.ng}{\tt olayeni@oauife.edu.ng}\\

\end{center}

\begin{abstract}
\noindent Let $(G_k)_{k\in\mathbb Z}$ be any sequence obeying the  recurrence relation of the Fibonacci numbers. We derive formulas for $\sum_{j=1}^n{G_{j + t}^6}$ and $\sum_{j=1}^n{(-1)^{j - 1}G_{j + t}^5(G_{j + t - 1} + G_{j + t + 1})}$, thereby extending the results of Ohtsuka and Nakamura who found simple formulas for $\sum_{j=1}^n{F_j^6}$ and $\sum_{j=1}^n{L_j^6}$, where $F_k$ and $L_k$ are the $k$th Fibonacci and Lucas numbers. We also evaluate $\sum_{j = 1}^n {G_{j + t}^3 G_{j + t + 1}^3 } $ and $\sum_{j = 1}^n {G_{j + t - 1}^2 G_{j + t} G_{j + t + 1} G_{j + t + 2}^2 } $, of which the results of Treeby are particular cases.

\end{abstract}


\section{Introduction}
Let $F_k$ and $L_k$ be the $k$th Fibonacci and Lucas numbers. 

In the year 2010, Ohtsuka and Nakamura~\cite{ohtsuka2010} derived the following formulas for the sum of the sixth powers of Fibonacci and Lucas numbers.
\begin{gather}
\sum_{j = 1}^n {F_j^6 }  = \frac{{F_n^5 F_{n + 3}  + F_{2n} }}{4}\label{eq.rn4ckio},\\
\sum_{j = 1}^n {L_j^6 }  = \frac{{L_n^5 L_{n + 3}  + 125F_{2n} }}{4} - 32\label{eq.fxdxuyy}.
\end{gather}

We extend~\eqref{eq.rn4ckio} and~\eqref{eq.fxdxuyy} as follows:
\begin{gather*}
\sum_{j = 1}^n {F_{j + t}^6 }  = \frac{{F_{n + t}^5 F_{n + t + 3}  - F_t^5 F_{t + 3} }}{4} + \frac{{F_{2n + 2t}  - F_{2t} }}{4},\\
\sum_{j = 1}^n {L_{j + t}^6 }  = \frac{{L_{n + t}^5 L_{n + t + 3}  - L_t^5 L_{t + 3} }}{4} + \frac{{125(F_{2n + 2t}  - F_{2t} )}}{4}.
\end{gather*}
Let $(G_j)_{j\in\mathbb Z}$ be the generalized Fibonacci sequence having the same recurrence relation as the Fibonacci and Lucas sequences but starting with arbitrary initial values; that is, let
\begin{equation}
G_j  = G_{j - 1}  + G_{j - 2},\quad (j \ge 2),
\end{equation}
where $G_0$ and $G_1$ are arbitrary numbers (usually integers) not both zero; with
\begin{equation*}
G_{-j} = G_{-(j - 2)} - G_{-(j - 1)}.
\end{equation*} 

Our goal in this note is to derive the following identity:
\begin{equation*}
\begin{split}
\sum_{j = 1}^n {G_{j + t}^6 } & = \frac{G_{n + t}^5 G_{n + t + 3}  - G_t^5 G_{t + 3}}4\\ 
 &\qquad + \frac{e ^2 \left( {G_{n + t} (G_{n + t + 1}  + G_{n + t - 1} ) - G_t (G_{t + 1}  + G_{t - 1} )} \right)}4,
\end{split}
\end{equation*}
where $e=G_0^2 - G_1^2 + G_0G_1$ here and throughout this paper; and $t$ and $n$ are any integers.

We will also derive the following results:
\begin{align*}
\sum_{j = 1}^n {( - 1)^{j - 1} F_j^5 L_j }  &= \frac{{( - 1)^{n + 1} }}{2}F_n^2 F_{n + 1}^2 \left( {F_{n + 1}^2  - F_n F_{n + 3} } \right),\\
\sum_{j = 1}^n {( - 1)^{j - 1} L_j^5 F_j }  &= \frac{{( - 1)^{n + 1} }}{{10}}L_n^2 L_{n + 1}^2 \left( {L_{n + 1}^2  - L_n L_{n + 3} } \right) + \frac{{14}}{5}.
\end{align*}
More generally we will evaluate $\sum_{j = 1}^n {( - 1)^{j - 1} G_{j + t}^5 \left( {G_{j + t + 1}  + G_{j + t - 1} } \right)}$.

Finally, we will establish
\begin{equation*}
\sum_{j = 1}^n {G_{j + t}^3 G_{j + t + 1}^3 }  = \frac{1}{4}\left( {G_{n + t}^2 G_{n + t + 1}^2 G_{n + t + 2}^2  - G_t^2 G_{t + 1}^2 G_{t + 2}^2 } \right),
\end{equation*}
and
\begin{equation*}
\sum_{j = 1}^n {\frac{1}{{G_{j + t - 1}^2 G_{j + t} G_{j + t + 1} G_{j + t + 2}^2 }}}  = \frac{1}{4}\left( {\frac{1}{{G_t^2 G_{t + 1}^2 G_{t + 2}^2 }} - \frac{1}{{G_{n + t}^2 G_{n + t + 1}^2 G_{n + t + 2}^2 }}} \right);
\end{equation*}
thereby generalizing the results of Treeby~\cite{treeby16} who found formulas for $\sum_{j = 1}^n {F_j^3 F_{j + 1}^3 }$ and $\sum_{j = 1}^n {L_j^3 L_{j + 1}^3 }$.

We require the following telescoping summation identities:
\begin{equation}\label{scope1}
\sum_{j = 1}^n {(f_{j + 1}  - f_j )}  = f_{n + 1}  - f_1,
\end{equation}
and
\begin{equation}\label{scope2}
\sum_{j = 1}^n {( - 1)^{j - 1} (f_{j + 1}  + f_j )}  = ( - 1)^{n + 1} f_{n + 1}  + f_1.
\end{equation}
We will also make use of the following well-known identities:
\begin{equation}\label{vajda1}
G_r G_{r + 2}  = G_{r + 1}^2  + ( - 1)^re,\quad\text{\cite[Formula 28]{vajda}},
\end{equation}
and
\begin{equation}\label{vajda2}
G_{r + s}  + ( - 1)^s G_{r - s}  = L_s G_r,\quad\text{\cite[Formula 10a]{vajda}};
\end{equation}
valid for all integers $r$ and $s$.

\section{Main results}
In this section, we evaluate $\sum_{j = 1}^n {G_{j + t}^6 }$. First we state two lemmata.
\begin{lemma}
If $r$ is an integer, then
\begin{equation}\label{eq.ytugbqv}
G_{r - 2} G_{r + 2}=G_r^2  + ( - 1)^r e.
\end{equation}
\end{lemma}
\begin{proof}
We have
\begin{equation*}
\begin{split}
G_{r - 2} G_{r + 2}  &= \left( {G_r  - G_{r - 1} } \right)\left( {G_r  + G_{r + 1} } \right)\\
& = G_r^2  + (G_r G_{r + 1}  - G_{r - 1} G_r ) - G_{r - 1} G_{r + 1} \\
& = G_r^2  + G_r^2  - G_{r - 1} G_{r + 1} \\
& = G_r^2  + G_r^2  - \left( {G_r^2  - ( - 1)^r e } \right),\quad\text{by~\eqref{vajda1}},\\
& = G_r^2  + ( - 1)^r e.
\end{split}
\end{equation*}
\end{proof}
\begin{lemma}
If $j$ and $t$ are any integers, then
\begin{equation}\label{gsquared}
\sum_{j=1}^n{G_{j + t}^2}= {G_{n + t} G_{n + t + 1}  - G_t G_{t + 1} } .
\end{equation}
\end{lemma}
\begin{proof}
Use~\eqref{scope1} to sum the identity
\begin{equation*}
G_{j + t}^2=G_{j + t + 1}G_{j + t} - G_{j + t}G_{j + t -1}.
\end{equation*}
\end{proof}
\begin{theorem}
If $n$ and $t$ are integers, then
\begin{equation}\label{sixth}
\begin{split}
\sum_{j = 1}^n {G_{j + t}^6 } & = \frac{G_{n + t}^5 G_{n + t + 3}  - G_t^5 G_{t + 3}}4\\ 
 &\qquad + \frac{e ^2 \left( {G_{n + t} (G_{n + t + 1}  + G_{n + t - 1} ) - G_t (G_{t + 1}  + G_{t - 1} )} \right)}4.
\end{split}
\end{equation}

\end{theorem}
In particular,
\begin{gather}
\sum_{j = 1}^n {F_{j + t}^6 }  = \frac{{F_{n + t}^5 F_{n + t + 3}  - F_t^5 F_{t + 3} }}{4} + \frac{{F_{2n + 2t}  - F_{2t} }}{4},\\
\sum_{j = 1}^n {L_{j + t}^6 }  = \frac{{L_{n + t}^5 L_{n + t + 3}  - L_t^5 L_{t + 3} }}{4} + \frac{{125(F_{2n + 2t}  - F_{2t} )}}{4}.
\end{gather}

\begin{proof}
Use \eqref{scope1} with $f_j  = G_{j + t - 3} G_{j + t - 2} G_{j + t - 1} G_{j + t} G_{j + t + 1} G_{j + t + 2} $ to obtain
\begin{equation*}
\begin{split}
&\sum_{j = 1}^n {G_{j + t - 2} G_{j + t - 1} G_{j + t} G_{j + t + 1} G_{j + t + 2} \left( {G_{j + t + 3}  - G_{j + t - 3} } \right)} \\
&\qquad\qquad = G_{n + t - 2} G_{n + t - 1} G_{n + t} G_{n + t + 1} G_{n + t + 2} G_{n + t + 3}  - G_{t - 2} G_{t - 1} G_t G_{t + 1} G_{t + 2} G_{t + 3} ;
\end{split}
\end{equation*}
that is
\begin{equation*}
\begin{split}
&\sum_{j = 1}^n {G_{j + t} (G_{j + t + 3}  - G_{j + t - 3} )\red{G_{j + t - 1} G_{j + t + 1}}\blue{ G_{j + t - 2} G_{j + t + 2}} }\\
&\qquad\qquad = G_{n + t}G_{n + t + 3}\red{G_{n + t - 1}G_{n + t + 1}}\blue{G_{n + t - 2}  G_{n + t + 2}}   - G_tG_{t + 3}\blue{G_{t - 1}G_{t + 1}}\red{G_{t - 2}G_{t + 2}}  ;
\end{split}
\end{equation*}
which becomes
\begin{equation*}
\begin{split}
&4\sum_{j = 1}^n {G_{j + t}^2 \left( {G_{j + t}^2  - ( - 1)^{j + t} e } \right)\left( {G_{j + t}^2  + ( - 1)^{j + t} e } \right)}\\ 
 &\qquad= G_{n + t} G_{n + t + 3} \left( {G_{n + t}^2  - ( - 1)^{n + t} e } \right)\left( {G_{n + t}^2  + ( - 1)^{n + t} e } \right)\\
&\qquad\qquad + G_t G_{t + 3} \left( {G_t^2  - ( - 1)^t e } \right)\left( {G_t^2  + ( - 1)^t e } \right);
\end{split}
\end{equation*}
with the use of \eqref{vajda2} with $s=3$,~\eqref{vajda1}, and~\eqref{eq.ytugbqv}.

Thus,
\begin{equation*}
4\sum_{j = 1}^n {\left( {G_{j + t}^6  - e ^2 G_{j + t}^2 } \right)}  = G_{n + t} G_{n + t + 3} \left( {G_{n + t}^4  - e ^2 } \right) - G_t G_{t + 3} \left( {G_t^4  - e ^2 } \right);
\end{equation*}
so that
\begin{equation*}
\sum_{j = 1}^n {G_{j + t}^6 } = \frac{{G_{n + t} G_{n + t + 3} (G_{n + t}^4  - e ^2 )}}{4} - \frac{{G_t G_{t + 3} (G_t^4  - e ^2 )}}{4} + e^2\sum_{j=1}^n{G_{j + t}^2};
\end{equation*}
which upon inserting~\eqref{gsquared} gives
\begin{equation}\label{eq.uaal7by}
\begin{split}
\sum_{j = 1}^n {G_{j + t}^6 } & = \frac{{G_{n + t} G_{n + t + 3} (G_{n + t}^4  - e ^2 )}}{4} - \frac{{G_t G_{t + 3} (G_t^4  - e ^2 )}}{4}\\
&\qquad\qquad + e ^2 \left( {G_{n + t} G_{n + t + 1}  - G_t G_{t + 1} } \right).
\end{split}
\end{equation} 
Simplication of~\eqref{eq.uaal7by} gives~\eqref{sixth} since
\begin{equation*}
4G_{s + 1}  - G_{s + 3}  = 3G_{s + 1}  - G_{s + 2}  = 2G_{s + 1}  - G_s  = G_{s + 1}  + G_{s - 1}.
\end{equation*}
\end{proof}

\section{Additional results}
We proceed to evaluate a couple of degree $6$ sums.
\begin{lemma}
If $j$ is an integer, then
\begin{align}
G_j^2 G_{j + 1}^2 G_{j + 2}^2  + G_{j - 1}^2 G_j^2 G_{j + 1}^2  &= 2G_j^4 G_{j + 1}^2  + 2G_{j + 1}^4 G_j^2,\label{eq.ltf416x}\\
G_j^2 G_{j + 1}^2 G_{j + 2}^2  - G_{j - 1}^2 G_j^2 G_{j + 1}^2  &= 4G_j^3 G_{j + 1}^3\label{eq.trkcg03}.
\end{align}
\end{lemma}
\begin{proof}
Square and add the identities
\begin{equation}
G_{j + 2} - G_{j -1}=2G_j\label{eq.vr34gut}
\end{equation}
and
\begin{equation}
G_{j + 2} + G_{j -1}=2G_{j + 1},\label{eq.w8ooqde}
\end{equation}
to obtain
\begin{equation}
G_{j + 2}^2  + G_{j - 1}^2  = 2\left( {G_j^2  + G_{j + 1}^2 } \right),
\end{equation}
and hence~\eqref{eq.ltf416x}. Multiplication of~\eqref{eq.vr34gut} and~\eqref{eq.w8ooqde} gives
\begin{equation*}
G_{j + 2}^2-G_{j - 1}^2=4G_jG_{j + 1},
\end{equation*}
from which~\eqref{eq.trkcg03} follows.
\end{proof}
\begin{proposition}
If $n$ and $t$ are integers, then
\begin{equation}
\begin{split}
\sum_{j = 1}^n {( - 1)^{j - 1} G_{j + t}^5 \left( {G_{j + t + 1}  + G_{j + t - 1} } \right)} & = \frac{{( - 1)^{n + 1} }}{2} {G_{n + t}^2 G_{n + t + 1}^2 G_{n + t + 2}^2  + \frac12G_t^2 G_{t + 1}^2 G_{t + 2}^2 } \\
&\qquad + ( - 1)^n G_{n + t + 1}^4 G_{n + t}^2  - G_{t + 1}^4 G_t^2 .
\end{split}
\end{equation}
\end{proposition}
\begin{proof}
Write $j + t$ for $j$ in~\eqref{eq.ltf416x}, multiply through by $(-1)^{j - 1}$ and sum over $j$, using~\eqref{scope2}.
\end{proof}
In particular,
\begin{equation}
\begin{split}
\sum_{j = 1}^n {( - 1)^{j - 1} F_{j + t}^5 L_{j + t}} & = \frac{{( - 1)^{n + 1} }}{2} {F_{n + t}^2 F_{n + t + 1}^2 F_{n + t + 2}^2  + \frac12F_t^2 F_{t + 1}^2 F_{t + 2}^2 } \\
&\qquad + ( - 1)^n F_{n + t + 1}^4 F_{n + t}^2  - F_{t + 1}^4 F_t^2 ,
\end{split}
\end{equation}
\begin{equation}
\begin{split}
\sum_{j = 1}^n {( - 1)^{j - 1} L_{j + t}^5 F_{j + t}} & = \frac{{( - 1)^{n + 1} }}{10} {L_{n + t}^2 L_{n + t + 1}^2 L_{n + t + 2}^2  + \frac1{10}L_t^2 L_{t + 1}^2 L_{t + 2}^2 } \\
&\qquad + \frac{( - 1)^n}5 L_{n + t + 1}^4 L_{n + t}^2  - \frac15L_{t + 1}^4 L_t^2 ;
\end{split}
\end{equation}
with the special cases
\begin{align}
\sum_{j = 1}^n {( - 1)^{j - 1} F_j^5 L_j }  &= \frac{{( - 1)^{n + 1} }}{2}F_n^2 F_{n + 1}^2 \left( {F_{n + 1}^2  - F_n F_{n + 3} } \right),\\
\sum_{j = 1}^n {( - 1)^{j - 1} L_j^5 F_j }  &= \frac{{( - 1)^{n + 1} }}{{10}}L_n^2 L_{n + 1}^2 \left( {L_{n + 1}^2  - L_n L_{n + 3} } \right) + \frac{{14}}{5}.
\end{align}

\begin{proposition}
If $n$ and $t$ are integers, then
\begin{align}
\sum_{j = 1}^n {G_{j + t}^3 G_{j + t + 1}^3 }  &= \frac{1}{4}\left( {G_{n + t}^2 G_{n + t + 1}^2 G_{n + t + 2}^2  - G_t^2 G_{t + 1}^2 G_{t + 2}^2 } \right),\label{eq.m4e3d1l}\\
\sum_{j = 1}^n {\frac{1}{{G_{j + t - 1}^2 G_{j + t} G_{j + t + 1} G_{j + t + 2}^2 }}}  &= \frac{1}{4}\left( {\frac{1}{{G_t^2 G_{t + 1}^2 G_{t + 2}^2 }} - \frac{1}{{G_{n + t}^2 G_{n + t + 1}^2 G_{n + t + 2}^2 }}} \right)\label{eq.wg4kr52}.
\end{align}
\end{proposition}
\begin{proof}
Identity~\eqref{eq.m4e3d1l} is a consequence of summing~\eqref{eq.trkcg03}. Identity~\eqref{eq.wg4kr52} follows from taking the reciprocal of each term of~\eqref{eq.trkcg03} and summing.
\end{proof}
In particular,
\begin{align}
\sum_{j = 1}^n {F_{j + t}^3 F_{j + t + 1}^3 }  &= \frac{1}{4}\left( {F_{n + t}^2 F_{n + t + 1}^2 F_{n + t + 2}^2  - F_t^2 F_{t + 1}^2 F_{t + 2}^2 } \right),\\
\sum_{j = 1}^n {L_{j + t}^3 L_{j + t + 1}^3 }  &= \frac{1}{4}\left( {L_{n + t}^2 L_{n + t + 1}^2 L_{n + t + 2}^2  - L_t^2 L_{t + 1}^2 L_{t + 2}^2 } \right);
\end{align}
and
\begin{align}
\sum_{j = 1}^n {\frac{1}{{F_{j + t - 1}^2 F_{j + t} F_{j + t + 1} F_{j + t + 2}^2 }}}  &= \frac{1}{4}\left( {\frac{1}{{F_t^2 F_{t + 1}^2 F_{t + 2}^2 }} - \frac{1}{{F_{n + t}^2 F_{n + t + 1}^2 F_{n + t + 2}^2 }}} \right),\\
\sum_{j = 1}^n {\frac{1}{{L_{j + t - 1}^2 L_{j + t} L_{j + t + 1} L_{j + t + 2}^2 }}}  &= \frac{1}{4}\left( {\frac{1}{{L_t^2 L_{t + 1}^2 L_{t + 2}^2 }} - \frac{1}{{L_{n + t}^2 L_{n + t + 1}^2 L_{n + t + 2}^2 }}} \right);
\end{align}
with the special cases
\begin{align}
\sum_{j = 1}^n {F_j^3 F_{j + 1}^3 }  &= \frac{1}{4}F_n^2 F_{n + 1}^2 F_{n + 2}^2, \label{eq.iowhr79}\\
\sum_{j = 1}^n {L_j^3 L_{j + 1}^3 }  &= \frac{1}{4}L_n^2 L_{n + 1}^2 L_{n + 2}^2  - 9\label{eq.wzvx9al};
\end{align}
and
\begin{align}
\sum_{j = 1}^n {\frac{1}{{F_j^2 F_{j + 1} F_{j + 2} F_{j + 3}^2 }}}  = \frac{1}{4}\left( {\frac{1}{4} - \frac{1}{{F_{n + 1}^2 F_{n + 2}^2 F_{n + 3}^2 }}} \right),\\
\sum_{j = 1}^n {\frac{1}{{L_j^2 L_{j + 1} L_{j + 2} L_{j + 3}^2 }}}  = \frac{1}{4}\left( {\frac{1}{144} - \frac{1}{{L_{n + 1}^2 L_{n + 2}^2 L_{n + 3}^2 }}} \right).
\end{align}
Identities~\eqref{eq.iowhr79} and~\eqref{eq.wzvx9al} were also obtained by Treeby~\cite{treeby16}.

\hrule

\hrule

\noindent Concerned with sequences: 
A000045, A000032

\end{document}